\documentclass{amsart}

\usepackage{amsfonts,amsthm,amsmath,amssymb,
dsfont}
\usepackage{enumerate}
\usepackage{graphicx}
\usepackage{color}
\usepackage[all]{xy}
\usepackage{subfigure}
\usepackage{tikz}
\usetikzlibrary[arrows,shapes,graphs,decorations.pathmorphing,backgrounds,positioning,fit,petri]
\usepackage{ amscd,   eufrak}
 \usepackage{mathrsfs}
 \usepackage{verbatim}
\usepackage{orcidlink}

\newtheorem{theorem}{Theorem}[section]
\newtheorem{lemma}[theorem]{Lemma}
\newtheorem{proposition}[theorem]{Proposition}

\newtheorem{definition}[theorem]{Definition}
\newtheorem{example}[theorem]{Example}
\newtheorem{remark}[theorem]{Remark}

 \newcommand{\R}{{\mathbb{R}}}
 \newcommand{\N}{{\mathbb{N}}}
 \newcommand{\fn}{\mathbb{E}\sp{1}}

\begin{document}


\title[level continuous fuzzy-valued functions]{A general framework for level continuous fuzzy-valued functions}

\author{J. J. Font, S. Macario, M. Sanchis }

\address{Institut Universitari
de Matem\`{a}tiques i Aplicacions de Castell\'{o} (IMAC),
Universitat Jaume I, Avda. Sos Baynat s/n, 12071, Castell\'{o}
(Spain).}

\email{font@uji.es,  macario@uji.es, sanchis@uji.es}



\begin{abstract}
In this paper we provide a general setting to deal with level continuous fuzzy-valued functions. Namely, we embed such functions into a product of spaces of real-valued functions of two variables satisfying certain types of left-continuity, right-continuity and monotonicity.
\end{abstract}





\maketitle

\section{Introduction}

Fuzzy analysis, which is based on the
notion of fuzzy number and fuzzy-valued function just as much as classical analysis is based on the concept
of real number and real-valued function, deviates  from classical analysis and needs its own specific development due to the different algebraic and structure theoretic properties of the usual real numbers and fuzzy numbers.

Fuzzy numbers are a powerful tool for modeling uncertainty and for processing vague or subjective information in mathematical models and have been
applied to many
areas, such as control theory, signal processing, pattern recognition, neural networks,
softcomputing, intelligent techniques, system theory, making decision, and so
on (see, e.g., \cite{LiYen} and \cite{Zimmermann}).

%
%
%

The set of fuzzy numbers, $\fn$, is usually endowed with several metrics or topologies most of which were introduced in the 1980s (\cite{PuriR:83}, \cite{PuriR:86}, \cite{kloeden:80}).  Perhaps the most studied being the convergence in the supremum metric $d_{\infty}$  and the other one being the level convergence ($\tau_{\ell}$) introduced by Kaleva and Seikkala (\cite{KalSeik:84}). It is well-known that the latter is weaker than the former.
As a consequence, the space of continuous fuzzy-valued functions $C(K,(\fn,\tau_{\ell}))$ contains $C(K,(\fn,d_{\infty}))$, where $K$ is a compact space.

Using  the Goetschel-Voxman's characterization of fuzzy numbers (\cite{GV}), it is possible to embed the set of fuzzy numbers $(\fn, d_{\infty})$ into a cone of a Banach space $X=\overline{C}([0,1])\times \overline{C}([0,1])$, where $\overline{C}([0,1])$ is the space of bounded left-continuous functions on $(0,1]$, right-continuous at $0$ and having right-limits on $[0,1)$ (see, e.g., \cite{WuMa:91}). This has been used for many purposes. Among them, we should mention the study of some properties of the functions in $C(K,\fn)$   as a particular case of the Banach-valued functions in $C(K,X)$ (see  \cite{WuMa:92-1}, \cite{WuMa:92}, \cite{chen},\cite{WuGong:00}, \cite{CHRL:16}, \cite{JardonSan:18}, \cite{FSS:2018}).

Another approach, when dealing with fuzzy-number-valued continuous functions, could be to study their properties by means of  real-valued functions of two variables. Each fuzzy number $u\in \fn$ provides two real-valued functions, $u^+$ and $u^-$, which stem from the Goetschel-Voxman's parametrization of fuzzy numbers (\cite{GV}).
As a consequence each
$f\in C(K,\fn)$ yields two functions
$f_i:[0,1]\times K\rightarrow \R$ ($i=1,2$), as follows
$$f_1(\lambda,t)=f(t)^-(\lambda) \quad\text{and}\quad f_2(\lambda,t)=f(t)^+(\lambda).$$

If we consider the cross-section functions in a separate way, both are located in well-studied function spaces. Indeed, for a fixed $t\in K$, the cross-section functions $f_i(\cdot,t):[0,1]\rightarrow \R$  are real-valued functions hosted in the Banach space  $\overline{C}([0,1])$ (see references in the previous paragraph). For a fixed $\lambda\in [0,1]$, the other cross-section functions $f_i(\lambda,\cdot):K\rightarrow\R$ lie in $C(K)$. But $f_1$ and $f_2$, as functions of two variables, will exhibit some particular features that are worthwhile looking into.


%

Similar representations appear in the literature (\cite{RL:18}, \cite{ACA:13}, \cite{chen}) but in the context of continuous fuzzy numbers, where both functions, $f_1$ and $f_2$, turn out to be jointly continuous on $[0,1]\times K$.  As we will show in this manuscript, this is not the case if we deal with (not necessarily continuous) fuzzy numbers. Indeed, in such more general context, the functions $f_1$ and $f_2$ need not be continuous in the first coordinate, even if we consider a function $f\in C(K,(\fn,d_{\infty}))$. Section 3 is devoted to constructing this representation and to listing the most important features of such functions, $f_1$ and $f_2$.

In Section 4, in order to handle with this general case, we introduce a new space of functions, $\overline{C}([0,1]\times K)$,
 which is neither a vector space nor even a cone since it is not closed for the sum.
We will prove that $C(K,(\fn,\tau_{\ell}))$
can be embedded isomorphically and isometrically into $\overline{C}([0,1]\times K)\times \overline{C}([0,1]\times K)$.

Finally, in Section 5, we take advantage of the above mentioned representation to provide certain differences between the spaces  $C(K,(\fn,\tau_{\ell}))$ and $C(K,(\fn,d_{\infty}))$, their behaviour with respect to the right limits being the key.


\section{Preliminaries}

Throughout the whole paper, the word space means a topological Tychonoff space. In particular,  $K$ stands for a compact (Hausdorff) space, unless otherwise stated.
As usual, $\R$ represents the real numbers and $\N$ the natural numbers. As topological spaces they are endowed with their usual topologies.
We begin with some standard definitions about fuzzy numbers to establish the notation throughout the paper.
Let $F(\R)$ denote the family of all fuzzy subset,  that is, mappings $u:\R\rightarrow [0,1]$. For $u\in F(\R)$ and $\lambda\in [0,1]$, the $\lambda$-level set
of $u$ is defined by
$$
[u]\sp{\lambda}:=\{x\in\R \ :\ u(x)\geq \lambda\}, \quad
\lambda\in ]0,1] ,\quad [u]\sp{0}:={\rm cl}\sb{\R}\{x\in\R\ :\ u(x)>0\},
$$
where ${\rm cl}\sb{\R}A$ stands for the closure of $A$ in $\R$. The fuzzy number space $\fn$ is the set of elements $u$ of
$F(\R)$ satisfying the following properties:

\begin{enumerate}[(1)]
 \item $u$ is normal, i.e., its core is nonempty, which is to say that there exists an $x\sb{0}\in\R$ with $u(x\sb{0})=1$;
\item $u$ is convex, i.e., $u(\lambda x + (1-\lambda)y)\geq \min
\left\{u(x),u(y)\right\}$  for all $x,y\in \R, \lambda\in [0,1]$;
\item $u$ is upper-semicontinuous;
\item $[u]\sp{0}$, the support of $u$, is a compact set in $\R$.
\end{enumerate}

Notice that if $u\in\fn$, then the $\lambda$-level set
$[u]\sp{\lambda}$ of $u$ is a compact interval for each
$\lambda\in [0,1]$. We denote by $u^-(\lambda)$ and $u^+(\lambda)$ the endpoints of such interval; that is $[u]\sp{\lambda}=[u^-(\lambda),u^+(\lambda)]$.  Every real
number $r$ can be considered a fuzzy number since  $r$ can be
identified with the fuzzy number $\tilde{r}$ defined as

\[
\tilde{r}(t):=\left\{
\begin{array}{lc}
1 & \text{if } t=r,\\[1ex]
0 & \text{if } t\neq r.
\end{array}\right.
\]

We can now state the characterization of fuzzy numbers provided by Goetschel and Voxman (\cite{GV}):

\begin{theorem}\label{GW1}
Let $u\in\fn$ and $[u]\sp{\lambda}=[u^-(\lambda),u^+(\lambda)]$, $\lambda\in
 [0,1]$.
Then the pair of functions $u\sp{-}(\lambda)$ and
$u\sp{+}(\lambda)$ has the following properties:

\begin{enumerate}[(i)]
 \item $u\sp{-}(\lambda)$ is a bounded left continuous nondecreasing function
on $\mathopen(0,1\mathclose]$;

\item  $u\sp{+}(\lambda)$ is a bounded left continuous nonincreasing function
on $\mathopen (0,1\mathclose]$;

\item $u\sp{-}(\lambda)$ and $u\sp{+}(\lambda)$ are right continuous at
$\lambda=0$;

\item $u\sp{-}(1)\leq u\sp{+}(1)$.
\end{enumerate}
Conversely, if a pair of functions $\alpha (\lambda)$ and
$\beta(\lambda)$ satisfy the above conditions (i)-(iv), then there
exists a unique $u\in\fn$ such that
$[u]\sp{\lambda}=\mathopen[\alpha(\lambda),\beta(\lambda)\mathclose]$
for each $\lambda\in [0,1]$.
\end{theorem}

Fuzzy numbers, $u$, whose corresponding functions $u^-, u^+$ are continuous will be called continuous fuzzy numbers and we shall denote them as
$$\fn_{c}:=\{ u\in\fn\ :\ u^-,\; u^+ \ \text{are continuous}\}.$$

Given $u,v\in \fn$ and $k\in \mathbb{R}$, we can define
$u+v$ and $ku$ using interval arithmetic (see, e.g., \cite{DK:1994}), and
it is well-known that $\fn$ endowed with these two natural operations
is not a vector space but a cone.

We can endow $\fn$ with the supremum metric:

\begin{definition}{\rm \cite{GV}}\label{GW2}
 For $u, v\in\fn$, we can define
$$
d\sb{\infty}(u,v):=\sup\sb{\lambda\in [0,1]}\max\left\{|u\sp{-}(\lambda)-v\sp{
- } (\lambda)|, |u\sp{+}(\lambda)-v\sp{+} (\lambda)|\right \}.
$$

\end{definition}

It is worth mentioning that
$$
d\sb{\infty}(u,v)=\sup\sb{\lambda\in [0,1]}\left\{d_H([u]^{\lambda},[v]^{\lambda}) \right\},
$$
where $d_H$ stands for the Hausdorff metric. Furthermore,
$(\fn , d\sb{\infty})$ is known to be a complete metric
space (\cite{DK:1994}).



Although other metrics and topologies can be considered, we will focus on the level-convergence topology.

\begin{definition}
We say that a net $\{u_k\}_{k\in D} \subset \fn$ levelly converges to $u\in\fn$ if $\lim_k d_H([u_k]^{\lambda},[u]^{\lambda})=0$ for any $\lambda\in[0,1]$.
\end{definition}

It is easy to prove that the net $\{u_k\}_{k\in D} \subset \fn$ levelly converges to $u\in\fn$, if and only if, $\lim_k u_k^-(\lambda)=u^-(\lambda)$ and
 $\lim_k u_k^+(\lambda)=u^+(\lambda)$, for each $\lambda\in [0,1]$.

Fang and Huang described in \cite{FH:03} (see also \cite{FoMiSan}) the topology for this convergence  where the basic open sets can be written as
$$\text{\small $V(u,\{\lambda_1,\ldots,\lambda_n\},\epsilon):=
\{v\in\fn\; :\; \max_{1\leq j\leq n} \{|v^-(\lambda_j)-u^-(\lambda_j)|,
|v^+(\lambda_j)-u^+(\lambda_j)|\}<\epsilon\}$}$$
for $u\in\fn$, $\{\lambda_1,\ldots,\lambda_n\}\subset[0,1]$ and $\epsilon>0$.

\medskip
 We shall denote by $\fn_{\ell}$ the set of fuzzy numbers endowed with the topology of level convergence and by  $\fn_{\infty}$ when it carries the topology induced by the metric $d_{\infty}$.

\section{A representation for fuzzy-valued continuous functions}

As mentioned in the Introduction, it is clear that, on $\fn$,  the topology of level-convergence $\tau_{\ell}$ is coarser than the topology induced by the metric $d_{\infty}$, thus resulting in $C(K,\fn_{\infty})\subset C(K,\fn_{\ell})$. This inclusion is proper as the following example, borrowed from \cite{FH:04}, shows.

\begin{example}[Example 5.1 in \cite{FH:04}]\label{ex:levelcont1}
{\em Let $f:[0,1]\rightarrow \fn$ be a fuzzy-valued function defined as
\begin{align*}
f(t)(x)&=\begin{cases}
\frac{1}{2}+x^{1/t} & \text {if $0\leq x \leq \left(\frac{1}{2}\right)^t$}\\[2ex]
1  & \text {if $\left(\frac{1}{2}\right)^t\ <  x \leq 1$}\\[2ex]
0  & \text {otherwise,}
\end{cases} \quad t\in(0,1]\\[2ex]
f(0)(x)&=\begin{cases}
\frac{1}{2} & \text {if $0\leq x < 1$}\\[2ex]
1  & \text {if $ x=1$}\\[2ex]
0  & \text {otherwise,}
\end{cases} \\[2ex]
\end{align*}
Then, for each $\lambda\in[0,1]$, we have $f_2(\lambda,t):=f(t)^+(\lambda)=1$ for all $t\in [0,1]$ and
\begin{align*}
f_1(\lambda,t):=f(t)^-(\lambda)&=\begin{cases}
0& \text {if $0\leq \lambda \leq \frac{1}{2}$}\\[2ex]
\left(\lambda-\frac{1}{2}\right)^t  & \text {if $\frac{1}{2} < \lambda  \leq 1,$}
\\[2ex]
\end{cases} \quad t\in(0,1],
\\[2ex]
f_1(\lambda,0):=f(0)^-(\lambda)&=\begin{cases}
0 & \text {if $0\leq \lambda \leq \frac{1}{2}$}\\[2ex]
1  & \text {if $ \frac{1}{2}<\lambda\leq 1.$}\\[2ex]
\end{cases}
\end{align*}
It is clear that $f(t)$ is level-continuous on $[0,1]$ but it is not $d_{\infty}-continuous$ at $t=0$.}
\end{example}

%
%

In this section we provide a representation of the functions in $C(K,\fn_{\ell})$ and $C(K,\fn_{\infty})$ which seems to be useful to study their properties. Such representation stems from the Goetschel-Voxman's characterization of fuzzy numbers (see Theorem~\ref{GW1}).


\begin{definition}\label{f1f2}
Let $K$ be a space. Let $f:K\rightarrow \fn$ be a fuzzy-valued function. We define the two following functions
$f_1,f_2:[0,1]\times K\rightarrow\R$ as
\[
\begin{aligned}
f_1(\lambda,t)&:=[f(t)^-](\lambda),\\
f_2(\lambda,t)&:=[f(t)^+](\lambda).
\end{aligned}
\]
\end{definition}

\bigskip
From the Goetschel-Voxman's characterization, we know that, for each $t\in K$, the cross-section functions $f_i(\cdot,t):[0,1]\rightarrow\R$, $i=1,2$, satisfy
properties $(i)-(iv)$ in Theorem \ref{GW1}. We are also interested in what happens to the other cross-section functions, that is, $f_i(\lambda,\cdot):K\rightarrow\R$, $i=1,2$, for a fixed $\lambda \in [0,1]$. We will deal with two cases, namely, when $f_i$, $i=1,2$, come from a function $f\in C(K,\fn_{\ell})$  or when they do from  $f\in C(K,\fn_{\infty})$. It is apparent that, in both cases, we obtain continuous functions $f_i:K\rightarrow \R$, $i=1,2$, but we need to take a closer look at which type of continuity we get.

Thus, if $f:K\rightarrow \fn_{\infty}$ is a continuous function, we have that, given $t_0\in K$ and $\epsilon>0$, there exists an open neighbourhood $ V_{t_0}$ of $t_0\in K$ such that, for every $t\in V_{t_0}$, we have
$d_{\infty}(f(t),f(t_0))<\epsilon$,
which means that
\[
\begin{aligned}
|f_1(\lambda,t)-f_1(\lambda,t_0)|&<\epsilon \\
|f_2(\lambda,t)-f_2(\lambda,t_0)|&<\epsilon \\
\end{aligned}\quad \text {for all $\lambda\in[0,1]$,\ $t\in V_{t_0}$.}
\]
It is worth noting that $V_{t_0}$ depends only on $t_0$ and $\epsilon$ but not on $\lambda$.
Then, we will say that
\begin{equation}\label{eq:uniformLambda}
\lim_{t\to t_0} f_i(\lambda,t)=f_i(\lambda,t_0)\quad\text{uniformly on $\lambda$.}
\end{equation}

On the other hand, if $f:K\rightarrow \fn_{\ell}$ is a continuous function, we will say that $f$ is level-continuous on $K$ and we will have that, given $t_0\in K$,
$\lambda_0\in[0,1]$  and $\epsilon>0$, there exists an open neighbourhood $V_{(\lambda_0,t_0)}$ of $t_0$ such that, for every $t\in V_{(\lambda_0,t_0)}$, we can write
$f(t)\in V(f(t_0),\lambda_0,\epsilon)$,
which means that, for such $\lambda_0\in[0,1]$ and $\epsilon>0$, there exists $V_{(\lambda_0,t_0)}$ with
\[
\begin{aligned}
|f_1(\lambda,t)-f_1(\lambda,t_0)|&<\epsilon, \\
|f_2(\lambda,t)-f_2(\lambda,t_0)|&<\epsilon, \\
\end{aligned}
\]
for all $t\in V_{(\lambda_0,t_0)}$.
We remark here that $V_{(\lambda_0,t_0)}$ depends  on $\lambda_0$, $t_0$ and $\epsilon$.
Then, we will say that
\begin{equation}\label{eq:pointconverg}
\lim_{t\to t_0} f_i(\lambda,t)=f_i(\lambda,t_0)\quad\text{pointwise on $\lambda$}.
\end{equation}

Hence, for both $\fn_{\infty}$ and $\fn_{\ell}$, we get, given a fixed $\lambda\in[0,1]$, two continuous functions in the second coordinate, $f_1(\lambda,\cdot),f_2(\lambda,\cdot):K\rightarrow \R$ which are also bounded since $K$ is  compact. Indeed,   we shall also obtain, in Section~\ref{sec:LCC},  that  $f_1$ and $f_2$ are  bounded as functions on $[0,1]\times K$.

 Summarizing all the above,  we can state that  every $f\in C(K, \fn_{\ell})$
  provides two real-valued functions $f_i(\lambda,t)$, $i=1,2$, defined on $[0,1]\times K$ with the following properties:
 for every $t\in K$,

\begin{enumerate}[(i)]
 \item $f_1(\cdot,t)$ is a bounded left continuous nondecreasing function
on $\mathopen (0,1\mathclose]$;

\item  $f_2(\cdot,t)$ is a bounded left continuous nonincreasing function
on $\mathopen (0,1\mathclose]$;

\item $f_1(\cdot,t)$ and $f_2(\cdot,t)$ are right continuous at
$\lambda=0$;

\item $f_1(1,t)\leq f_2(1,t)$;
\end{enumerate}
and, for every $\lambda\in[0,1]$ and $i=1,2$,
\begin{enumerate}
\item[(v)] $f_i(\lambda,\cdot)$ is a continuous function on $K$.

\end{enumerate}

Then, we have a representation for fuzzy-valued continuous functions, $f=(f_1,f_2)$, which share the properties (i)-(v) for both
$\fn_{\infty}$-valued continuous functions and for $\fn_{\ell}$-valued continuous functions.
We will see in Section~\ref{sec:diff} that this representation will allow us to show some differences between the spaces $C(K, \fn_{\ell})$ and $ C(K,\fn_{\infty})$ as well.

The parametric representation given by Goetschel-Voxman of the fuzzy numbers has been used for long to embed the fuzzy number space $(\fn,d_{\infty})$ into a Banach space, namely $X=\overline{C}([0,1])\times \overline{C}([0,1])$, where $\overline{C}([0,1])$ denotes the Banach space of bounded left-continuous functions on $(0,1]$, right continuous at $0$ and having right-limits on $[0,1)$, endowed with the supremum norm (see \cite{WuMa:91}). This embedding enables to represent the functions $ f\in C(K, \fn_{c})$ by functions $f\in C(K,X)$ and it has been used to obtain, among others, some results about integrability or measurability of such functions
(see, e.g., \cite{WuMa:92-1}, \cite{WuMa:92}), fuzzy optimization (\cite{Wu:04}), etc.

In \cite{ACA:13} it was defined an embedding of $C(K,\fn_{c})$ into $C([0,1]\times K)\times C([0,1]\times K)$, using  this representation, but it is worth mentioning that  the topology  inherited by $d_{\infty}$ and  the one from the level-convergence $\tau_{\ell}$ coincide on the set of continuous fuzzy numbers $\fn_{c}$ . We intend to provide a similar embedding but in the more general case when the fuzzy numbers are not continuous and $\fn$ carries the level-convergence topology. Firstly, we will need to study the behaviour of such  representation $(f_1,f_2)$ and the space where both functions $f_1$ and $f_2$ are naturally placed.

\section{The space of left continuous - continuous functions}
\label{sec:LCC}

Recall that $\fn_{c}$ stands for the set of continuous fuzzy numbers. Note that, if we consider continuous $\fn_{c}$-valued functions, the corresponding real-valued functions introduced in Definition \ref{f1f2} are continuous (see \cite{chen}) over the compact space $M=[0,1]\times K$ and a large number of classical results can be applied.
However, the situation is not so friendly when we consider  $\fn$-valued continuous functions. Indeed we need to introduce some new concepts in order to deal with this setback.

\begin{definition}
Let $K$ be a space. A function $f:[0,1]\times K\rightarrow \R$ is said to be
\begin{itemize}
\item
 {\em separately left continuous - continuous} at $(\lambda_0,t_0)\in(0,1]\times K$ when  $f(\cdot,t_0):[0,1]\rightarrow\R$ is left-continuous at each $\lambda\in(0,1]$, and
$f(\lambda_0,\cdot):K\rightarrow \R$ is continuous at each $t\in K$.

\medskip
\item {\em (jointly) left continuous - continuous} at $(\lambda_0,t_0)\in(0,1]\times K$ when,
for $\epsilon>0$ , there exist $\delta>0$ and an open neigbourhood  of $t_0\in K$, $V_{(\lambda_0,t_0)}$, such that
$$|f(\alpha,s)-f(\lambda_0,t_0)|<\epsilon, \quad\text{for all } (\alpha,s) \in (\lambda_0-\delta,\lambda_0)\times V_{(\lambda_0,t_0)}.$$
\end{itemize}
In a similar way, we can define {\em separately} (resp. {\em jointly  right continuous - continuous}) at $(\lambda_0,t_0)\in[0,1)\times K$.
\end{definition}

\bigskip
Let us study in more detail these new types of continuity.
The monotonicity in the first coordinate allows us to obtain a jointly left (resp. right) continuity - continuity result following the ideas applied to separately continuous functions that we can trace back to
Young \cite{Yo1910} (see also Kruse and Deely \cite{KD1969}). We say that a function $f:[0,1]\times K\rightarrow \R$ is monotone in the first coordinate if, fixed $t\in K$, $f(\cdot, t)$ is monotone. Such monotonicity could be different for distinct $t\in K$.

\begin{theorem}\label{th:jointlyLC}
Let $K$ be a space and let $f:[0,1]\times K\rightarrow \R$ be a separately left continuous - continuous at $(\lambda_0,t_0)\in(0,1]\times K$. If $f$ is monotone in the first coordinate, then $f$ is jointly left continuous - continuous at $(\lambda_0,t_0)$.
The same is true for separately right continuous - continuous at $(\lambda_0,t_0)\in[0,1)\times K$.
\end{theorem}
\begin{proof}
Let $(\lambda_0,t_0)\in(0,1]\times K$ and take $\epsilon>0$. Then, by the left-continuity of $f(\cdot,t_0)$ at $\lambda_0$, we have that
there exists $\delta>0$ such that
\begin{equation}\label{eq:leftC1}
|f(\lambda,t_0)-f(\lambda_0,t_0)|<\frac{\epsilon}{2},  \quad \text{ for } \lambda_0-\delta<\lambda<\lambda_0.
\end{equation}
Now, by the continuity of $f(\lambda_0-\delta,\cdot)$ and $f(\lambda_0,\cdot)$ at $t_0\in K$, we can find open neighbourhoods of $t_0\in K$,  $V_{(\lambda_0-\delta, t_0)}$ and $V_{(\lambda_0, t_0)}$, such that, for every $t\in W_{t_0}:=V_{(\lambda_0-\delta, t_0)}\cap V_{(\lambda_0, t_0)}$ we have
\begin{equation}\label{eq:Cont2}
|f(\lambda_0-\delta,t)-f(\lambda_0-\delta,t_0)|<\frac{\epsilon}{2},  \end{equation}
\begin{equation}\label{eq:Cont3}
|f(\lambda_0,t)-f(\lambda_0,t_0)|<\frac{\epsilon}{2}.
\end{equation}
Then, fix
$(\lambda,t)\in(\lambda_0-\delta,\lambda_0)\times W_{t_0}$ and assume that $f(\cdot,t)$ is nondecreasing.  We have, using \eqref{eq:Cont2} and \eqref{eq:leftC1},
\begin{align*}
f(\lambda,t)-f(\lambda_0,t_0)&\geq f(\lambda_0-\delta,t)-f(\lambda_0,t_0)\\
&=\Big(f(\lambda_0-\delta,t)-f(\lambda_0-\delta,t_0)\Big)+\Big(f(\lambda_0-\delta,t_0)-f(\lambda_0,t_0)\Big)\\
&>-\frac{\epsilon}{2} -\frac{\epsilon}{2}=-\epsilon
\end{align*}
On the other hand, using \eqref{eq:Cont3}
\begin{align*}
f(\lambda,t)-f(\lambda_0,t_0)&\leq f(\lambda_0,t)-f(\lambda_0,t_0)<\epsilon
\end{align*}

So, we have just proved that
$$-\epsilon <f(\lambda,t)-f(\lambda_0,t_0)<\epsilon,$$
which gives
$$|f(\lambda,t)-f(\lambda_0,t_0)|<\epsilon.$$

The case when $f(\cdot,t)$ is nonincreasing is similar, just reversing the inequalities above. This completes the proof for jointly left continuity - continuity.

The proof runs in a similar way for the jointly right continuity - continuity on $[0,1)\times K$.
\end{proof}

\medskip
Gathering all the information from above, we know that every $f\in C(K, \fn_{\ell})$ yields (see Definition \ref{f1f2}) two representation functions
$f_i(\lambda,t)$, $i=1,2$, defined on $[0,1]\times K$ with the following properties:
\begin{enumerate}[(i)]
\item $f_i(\lambda,t)$ are (jointly) left continuous - continuous at each $(\lambda,t)\in(0,1]\times K$.
\item $f_i(\lambda,t)$ are (jointly) right continuous - continuous at $(0,t)\in \{0\}\times K$.
\item $f_1(\cdot,t)$ is monotone nondecreasing  and $f_2(\cdot,t)$ is monotone nonincreasing.
\end{enumerate}

This motivates the following definition.

\begin{definition} \label{mainDefinition}
Let $K$ be a space.
Denote by $\overline{C}([0,1]\times K)$ the set of real-valued functions $f$ satisfying the following properties:
\begin{enumerate}[(i)]
\item $f$ is jointly left continuous - continuous at each $(\lambda,t)\in(0,1]\times K$;
\item $f$ is jointly right continuous - continuous at each $(0,t)\in\{0\}\times K$.
\item For every $t\in K$, $f(\cdot,t)$ is monotone.
\end{enumerate}
\end{definition}

Let us see that, when $K$ is a compact space,  $\overline{C}([0,1]\times K)$ is well located as a subset of the space of real-valued bounded functions.
Firstly, we notice that if $f \in \overline{C}([0,1]\times K)$ then,
for each $t\in K$, the cross-section function $f(\cdot,t):[0,1]\rightarrow \R$
is monotone and so has right-sided limit at each $\lambda_0\in [0,1)$.
We denote by $f(\lambda_0+,t)$ such right limit.

The real-valued functions defined on an interval $[a,b]$ which have left-sided limits at each $\lambda\in(a,b]$ and right-sided limits at each $\lambda\in[a,b)$ are called {\em regulated functions} (\cite{frankova:2019}). Indeed, when they are left-continuous and have right limits they are called \textit{c\`agl\`ad} functions (French: \textit{continue \`a gauche, limite \`a droite}).
It is well-known that if a real-valued function defined on a compact interval has one-sided limits, then it is bounded.
Adding a second coordinate we need to do an extra work and the monotonicity in the first coordinate helps to get the boundedness.

\begin{theorem}\label{th:bounded1}
Let $K$ be a compact space and
let $f\in \overline{C}([0,1]\times K)$.
Then $f$ is a bounded function.
\end{theorem}
\begin{proof}
Fix $(\lambda_0,t_0)\in [0,1]\times K$. If $\lambda_0\neq 1$ then, by the existence of right limit, we have that there exists
$\delta_{(\lambda_0,t_0)}^+ >0$ with
\begin{align}\label{eqRL1}
|f(\beta,t_0)-f(\lambda_0+,t_0)|<1 \quad \text{for $\lambda_0<\beta<\lambda_0 +\delta_{(\lambda_0,t_0)}^+$}.
\end{align}
If $\lambda_0\neq 0$, then, by the jointly left continuity - continuity at $(\lambda_0, t_0)$, we get $\delta_{(\lambda_0,t_0)}^->0$ and an open neighbourhood $V_{(\lambda_0,t_0)}$ of $t_0$
with
\begin{align}\label{eqRL2}
|f(\beta,s)-f(\lambda_0,t_0)|<1 \quad \text{for $\lambda_0-\delta_{(\lambda_0,t_0)}^-<\beta\leq \lambda_0$ and $s\in V_{(\lambda_0,t_0)}$.}
\end{align}

Also we can assume, by the continuity of the cross-sections $f(\lambda_0 +\delta_{(\lambda_0,t_0)}^+,\cdot)$ ($\lambda_0\neq 1$) and
$f(\lambda_0 -\delta_{(\lambda_0,t_0)}^-,\cdot)$ ($\lambda_0\neq  0$) at $t_0$, that there exists $V'_{(\lambda_0,t_0)}$ with
\begin{align}\label{eqRL3}
|f(\lambda_0 +\delta_{(\lambda_0,t_0)}^+,s)-f(\lambda_0 +\delta_{(\lambda_0,t_0)}^+,t_0)|&<1 \quad \text{for $s\in V'_{(\lambda_0,t_0)}$},\\
\label{eqRL4}
|f(\lambda_0 -\delta_{(\lambda_0,t_0)}^-,s)-f(\lambda_0 -\delta_{(\lambda_0,t_0)}^-,t_0)|&<1 \quad \text{for $s\in V'_{(\lambda_0,t_0)}$}.
\end{align}
Take $W_{(\lambda_0,t_0)}=V_{(\lambda_0,t_0)}\cap V'_{(\lambda_0,t_0)}$ and  define the following open sets in $[0,1]\times K$:
\begin{align*}
U_{(\lambda_0,t_0)}&:=(\lambda_0 -\delta_{(\lambda_0,t_0)}^-,\lambda_0 +\delta_{(\lambda_0,t_0)}^+)\times W_{(\lambda_0,t_0)}, \lambda_0\neq 0,1;\\
U_{(0,t_0)}&:=[0,\lambda_0 +\delta_{(0,t_0)}^+)\times W_{(0,t_0)};\\
U_{(1,t_0)}&:=(1 -\delta_{(\lambda_0,t_0)}^-,1]\times W_{(1,t_0)}.
\end{align*}

 Fix $(\beta,s)\in U_{(\lambda_0,t_0)}$. We have two possibilities
depending on which side of $\lambda_0$ is $\beta$.

\medskip
If $\lambda_0-\delta_{(\lambda_0,t_0)}^-<\beta\leq \lambda_0$ and $s\in W_{(\lambda_0,t_0)}$, by \eqref{eqRL2}, we have
\begin{align*}
|f(\beta,s)| &\leq |f(\beta,s)-f(\lambda_0,t_0)|+|f(\lambda,t_0)|<1
+|f(\lambda_0,t_0)|=\rho_{(\lambda_0,t_0)}^-.
\end{align*}

If  $\lambda_0<\beta<\lambda_0 +\delta_{(\lambda_0,t_0)}^+$ and $s\in W_{(\lambda_0,t_0)}$,  we  need the monotonicity for controlling the behaviour of $f(\beta,s)$.

Firstly,  we assume that $f(\cdot,s)$ is nondecreasing.
By \eqref{eqRL3} and \eqref{eqRL1}, we have
\begin{align*}
f(\beta,s) &\leq f(\lambda_0 +\delta_{(\lambda_0,t_0)}^+,s)=f(\lambda_0 +\delta_{(\lambda_0,t_0)}^+,s)-f(\lambda_0 +\delta_{(\lambda_0,t_0)}^+,t_0)\\
&+ f(\lambda_0 +\delta_{(\lambda_0,t_0)}^+,t_0 )-f(\lambda_0+,t_0)+f(\lambda_0+,t_0)<1+1+f(\lambda_0+,t_0)=q_{(\lambda_0,t_0)}^+.
\end{align*}
and, by \eqref{eqRL2},
\begin{align*}
f(\beta,s) &\geq f(\lambda_0,s)=f(\lambda_0,s)-f(\lambda_0,t_0)+ f(\lambda_0,t_0 )>-1+f(\lambda_0,t_0)=q_{(\lambda_0,t_0)}^-.
\end{align*}

Then,
$$q_{(\lambda_0,t_0)}^-\leq f(\beta,s)\leq q_{(\lambda_0,t_0)}^+$$

If $f(\cdot,s)$ is nonincreasing, then by reversing the inequalities, we have
\begin{align*}
f(\beta,s) &\geq -1-1+f(\lambda_0+,t_0)=p_{(\lambda_0,t_0)}^-,\\
f(\beta,s) &\leq 1+f(\lambda_0,t_0)=p_{(\lambda_0,t_0)}^+.
\end{align*}
and we obtain the bounds
$$p_{(\lambda_0,t_0)}^-\leq f(\beta,s)\leq p_{(\lambda_0,t_0)}^+.$$

Taking $\rho_{(\lambda_0,t_0)}^+=\max\{|q_{(\lambda_0,t_0)}^-|, |q_{(\lambda_0,t_0)}^+|,|p_{(\lambda_0,t_0)}^-|,|p_{(\lambda_0,t_0)}^+|\}$
we obtain, for every $(\beta,s) \in U_{(\lambda_0,t_0)}$,
$$|f(\beta,s)|\leq\rho_{(\lambda_0,t_0)}:=\max\{\rho_{(\lambda_0,t_0)}^-,\rho_{(\lambda_0,t_0)}^+\}.$$

Now,
 $\{  U_{(\lambda_0,t_0)}\}_{(\lambda_0,t_0)\in[0,1]\times K}$ is an open cover of the compact space $[0,1]\times K$. Then, we can find a finite subcover, say
 $\{  U_{(\lambda_k,t_k)}\}_{k=1}^{N}$.
 Let $\rho=\max\{ \rho_{(\lambda_k,t_k)}\ : 1\leq k\leq N\}$. Hence,
 for every $(\beta,s)\in [0,1]\times K$, we can find $1\leq k\leq N$ such that
 $(\beta,s)\in U_{(\lambda_k,t_k)}$, which means
 $$|f(\beta,s)|\leq\rho_{(\lambda_k,t_k)}\leq \rho$$
 and we are done.
\end{proof}

The following example shows that, in the above theorem, we cannot omit condition (iii) in Definition \ref{mainDefinition}:

\begin{example} {\rm
For each $n\geq 1$, let $f_n\colon [0,1]\to \mathbb{R}$ be the function defined as

$$
f_n(\lambda)=\left\{\begin{array}{cl}
0 & \quad \text{if} \quad \lambda \leq  \frac{1}{2},\medskip \\
-n^2\lambda+\frac{n^2+2n}{2} & \quad \text{if} \quad \frac{1}{2}< \lambda \leq \frac{1}{2} + \frac{1}{n}, \medskip \\
0 & \quad \text{if} \quad \frac{1}{2} + \frac{1}{n} < \lambda.
 \end{array}\right.
$$

Let $K=\left\{p\right\}\cup X$ be the Alexandroff one-point compactification of a countable (denumerably infinite) discrete space $X=\left\{a_n \, \mid \, n\geq 1\right\}$. Define the real-valued function $f$ on $[0,1]\times K$ by the rule: $f(\lambda,a_n)=f_n (\lambda)$ for any $(\lambda,a_n)\in [0,1]\times \left(K\setminus \{p\}\right)$ and $f(\lambda,p)=0$ for any $\lambda\in [0,1]$. By the definition of the functions $f_n$ ($n\geq 1$), $f$ is jointly right continuous-continuous at  $(0,t)\in \left\{0\right\}\times K$. Moreover, the function $f$ is unbounded: indeed, for all $n\geq 1$, the right-sided limit of $f_n$ at $\lambda=\frac{1}{2}$ is $n$.

We finish the example by showing that $f$ is jointly left continuous-continuous at each $(\lambda,t)\in (0,1]\times K$. It is apparent that we only need to consider the points $(\lambda,p)$ with $\lambda>\frac{1}{2}$. Fix $(\lambda_0,p)$ with $\lambda_0>\frac{1}{2}$. Since $\lambda_0>\frac{1}{2}$, there exists $n_0$ such that $\frac{1}{2}+\frac{1}{n_0}<\lambda_0$. Then $f_n$ vanishes on $(\frac{1}{2}+\frac{1}{n_0},\lambda_0]$ for all $n\geq n_0$.  Therefore, if $\lim_{k}\lambda_k=\lambda_0$ and $\lim_{k}a_{n_k}=p$, then   $\lim_k f(\lambda_k,a_{n_k})=0=f(\lambda_0,p)$. Thus, $f$ is jointly left continuous-continuous at $(\lambda_0,p)$.

}
\end{example}

For a space $X$, let $ B(X)$ be the family of bounded real-valued functions on $X$. It is well-known that it is a Banach space with the norm
$\|f\|_{\infty}:=\sup\{|f(x)|\ : \ x\in X\}$.

By the theorem above, we know that $\overline{C}([0,1]\times K)\subset B([0,1]\times K)$.
So we can endow this subset with the topology induced by the $\|\cdot\|_{\infty}$-norm. Moreover,
we will show that it is indeed a closed subset of $B([0,1]\times K)$.
\begin{proposition}\label{prop:unifconv1}
Let $\{f_n\}$ be a sequence in $\overline{C}([0,1]\times K)$ that $\| \cdot\|_{\infty}$-converges to a function $f\in B([0,1]\times K)$. Then
$f\in \overline{C}([0,1]\times K)$.
\end{proposition}
\begin{proof}
We need to show that
\begin{enumerate}[(i)]
\item $f$ is jointly left continuous - continuous at each $(\lambda,t)\in(0,1]\times K$;
\item $f$ is jointly right continuous - continuous at each  $(0,t)\in\{0\}\times K$;
\item For every $t\in K$, $f(\cdot,t)$ is monotone.
\end{enumerate}
(i) and (ii) are apparent. For (iii), fix $t_0\in K$. Each $f_n(\cdot,t_0)$ is a monotone function. We can split the set of indices into two sets:
\begin{align*}
A&:=\{n\in \N\ : f_n(\cdot,t_0) \text{ is monotone nondecreasing}\}\\
B&:=\{n\in \N\ : f_n(\cdot,t_0) \text{ is monotone nonincreasing}\}.\\
\end{align*}
If $\lambda,\beta\in [0,1]$ and $\lambda\leq \beta$, then
\begin{align*}
f_n(\lambda,t_0)&\leq f_n(\beta,t_0), \quad\text{ if $n\in A$},\\
f_n(\lambda,t_0)&\geq f_n(\beta,t_0), \quad\text{ if $n\in B$}.\\
\end{align*}
If both $A$ and $B$ are infinite, then the limit function $f(\cdot,t_0)$ is constant by taking limits in the inequalities above. If only one of them is infinite, then
$f(\cdot,t_0)$ is monotone as $f_n(\cdot,t_0)$ are.
\end{proof}

The proposition above shows that $\overline{C}([0,1]\times K)$ is a closed subset in $(B([0,1]\times K),\|\cdot\|_{\infty})$. So, it is also complete.

\medskip
It may happen that, for some $t \in K$, we have two functions of the space $\overline{C}([0,1]\times K)$ but with different type of monotonicity in the first coordinate, leading to the sum of these two functions not being a monotonic function at that point $t$. Thus, the space  $\overline{C}([0,1]\times K)$  is neither a vector space nor even a cone, since it is not closed for the sum.
  Nevertheless, it contains two special convex closed cones. One,  say
  $\overline{C}_{nd}([0,1]\times K)$, containing all the functions in $\overline{C}([0,1]\times K)$ for which $f(\cdot,t)$ is a nondecreasing function, for all $t\in K$. Another one, say $\overline{C}_{ni}([0,1]\times K)$, containing all the functions $f$ such that $f(\cdot, t)$ is a nonincreasing function, for all $t\in K$.

We know, by Definition~\ref{f1f2} and Theorem~\ref{th:bounded1}, that every function $f\in C(K,\fn_{\ell})$ provide two functions $f_1,f_2\in \overline{C}([0,1]\times K)$ and so they are bounded.
Hence we can define, for every $f,g\in C(K,\fn_{\ell})$, the metric
\[
D(f,g):=\sup_{t\in K}\sup_{\lambda\in [0,1]}\max\{|f_1(\lambda,t)-g_1(\lambda,t)|,|f_2(\lambda,t)-g_2(\lambda,t)|\}
\]
which can be rewritten as
\[
D(f,g)=\sup_{t\in K}d_{\infty}(f(t), g(t)),
\]
which is also the usual metric in $C(K,\fn_{\infty})$.

In \cite{ACA:13} the authors defined an embedding of $C(K,\fn_{c})$ into $C([0,1]\times K)\times C([0,1]\times K)$ using these functions $(f_1,f_2)$. We intend to provide a similar result but in the more general case when the fuzzy numbers are not continuous and $\fn$ carries the level-convergence topology.
Thus we can state (compare with Theorem~4.1 in \cite{ACA:13}):
\begin{theorem}\label{th:embedding}
The space $(C(K,\fn_{\ell}),D)$ can be embedded isomorphically and isometrically onto a closed convex cone of  $(\overline{C}([0,1]\times K)\times \overline{C}([0,1]\times K),D_{\infty})$, where the distance in this product space is given by
$$D_{\infty}((f,g),(h,j)):=\max\{ \|f-h\|_{\infty}, \|g-j\|_{\infty}\}.$$

\end{theorem}
\begin{proof}
We define
$$\Phi:C(K,\fn_{\ell})\longrightarrow \overline{C}([0,1]\times K)\times \overline{C}([0,1]\times K)$$
as $\Phi(f)=(f_1,f_2)$, being $(f_1,f_2)$ the representation of $f$ given in Definition~\ref{f1f2}.
Then, we have
\begin{itemize}
\item[(i)] $\Phi$ is one-to-one, since the representation $(f_1,f_2)$ of each function $f\in C(K,\fn_{\ell})$ is unique.

\item[(ii)]  $\Phi(C(K,\fn_{\ell}))$ is a closed convex cone.

Indeed, given $f,g\in C(K,\fn_{\ell})$ and $\mu,\eta\geq 0$, we have
$$\Phi(\mu f+\eta g)=\mu\Phi(f)+\eta\Phi(g)$$
since, taking $h=\mu f+\eta g$, we get

$h_1 =h(t)^-(\lambda)=\mu f(t)^-(\lambda)+\eta g(t)^-(\lambda)=
\mu f_1(\lambda,t)+\eta g_1(\lambda,t)$

and

$h_2 =h(t)^+(\lambda)=\mu f(t)^+(\lambda)+\eta g(t)^+(\lambda)=
\mu f_2(\lambda,t)+ \eta g_2(\lambda,t).$
\item[(iii)] $\Phi$ is an isometry. Indeed,
\begin{align*}
D(f,g)&=\sup_{t\in K}\sup_{\lambda\in [0,1]}\max\{|f_1(\lambda,t)-g_1(\lambda,t)|,|f_2(\lambda,t)-g_2(t,\lambda)|\}\\
&=\max\{\sup_{t\in K}\sup_{\lambda\in [0,1]}|f_1(\lambda,t)-g_1(\lambda,t)|,\sup_{t\in K}\sup_{\lambda\in [0,1]}|f_2(\lambda,t)-g_2(\lambda,t)|\}\\
&=\max\{ \| f_1-g_1\|_{\infty}, \| f_2-g_2\|_{\infty}\}
=D_{\infty}((f_1,f_2),(g_1,g_2)).
\end{align*}
%
\end{itemize}

\end{proof}

\section{Differences between $C(K, \fn_{\ell})$ and $C(K, \fn_{\infty})$}
\label{sec:diff}

In \cite{FoMiSan} it is noticed that, over the set of continuous fuzzy numbers $\fn_{c}$, both topologies, the one inherited by $d_{\infty}$ and  the one from the level-convergence $\tau_{\ell}$, coincide. So, $C(K,\fn_{c,\infty})=C(K,\fn_{c,\ell})$.
But in general, as $\fn_{\infty}\neq \fn_{\ell}$,  some differences in the behaviour of such functions are expected. How deep they are is an interesting question which we will outline here with some examples, being the key the right limit-continuous property (see \eqref{eq:RLCP} in Lemma~\ref{lem:jointRL} below).

First of all, one might think that if $f\in C(K, \fn_{\infty})$, then $f(t)$ is a continuous fuzzy number for every $t\in K$, in contrast to the level continuous case (see Example  \ref{ex:levelcont1}). However this is not the case as the following simple example shows.

\begin{example}
Let us fix a non-continuous fuzzy number, $u$.
We can now introduce the constant function $f(t)=u$ for all $t\in K$. It is apparent that $f\in C(K, \fn_{\infty})$, but $f(t)$ is not a continuous fuzzy number (for any $t\in K$).

\end{example}

In our next lemmas we focus on the behaviour of the functions in $C(K, \fn_{\infty})$ and in $C(K, \fn_{\ell})$ with respect to the right limits.

\begin{lemma}\label{lem:rightL1}
Let $f\in \overline{C}([0,1]\times K)$  and $\epsilon>0$.
For each $t_0\in K$ and every $\lambda_0\in[0,1)$, there exists
$\delta_{(\lambda_0,t_0)}>0$ such that, for all
$\lambda\in (\lambda_0,\lambda_0+\delta_{(\lambda_0,t_0)})$, we have
$$|f(\lambda+,t_0)-f(\lambda_0+,t_0)|<\epsilon$$
that is, $$\lim_{\lambda\to\lambda_0^+} f(\lambda+,t_0)=f(\lambda_0+,t_0).$$

\end{lemma}
\begin{proof}
Fix $t_0\in K$. We can find $\delta_{(\lambda_0,t_0)}>0$ such that,
for $\lambda\in (\lambda_0,\lambda_0+\delta_{(\lambda_0,t_0)})$,
\begin{align}
|f(\lambda,t_0)-f(\lambda_0+,t_0)|&<\epsilon.
\end{align}
Now assume that $f(\cdot,t_0)$ is nondecreasing. Then we have,
for every $\lambda\in (\lambda_0,\lambda_0+\delta_{(\lambda_0,t_0)})$,
$$0\leq f(\lambda+,t_0)-f(\lambda_0+,t_0)\leq f(\lambda,t_0)-f(\lambda_0+,t_0)<\epsilon,$$
which yields $$|f(\lambda+,t_0)-f(\lambda_0+,t_0)|<\epsilon.$$
If $f(\cdot,t_0)$ is nonincreasing then, by reversing the inequalities,
$$0\geq f(\lambda+,t_0)-f(\lambda_0+,t_0)\geq f(\lambda,t_0)-f(\lambda_0+,t_0)>-\epsilon$$
and we get the same conclusion.
\end{proof}
\begin{lemma}\label{lem:rightL2}
Let $f\in \overline{C}([0,1]\times K)$.  Assume that $f(\lambda, t)$ converges to  $f(\lambda, t_0)$ uniformly on $\lambda$ (see~\eqref{eq:uniformLambda}).
Then, for such $t_0\in K$ and for every $\epsilon>0$, there exist an open neighbourhood $V_{t_0}$ of $t_0$ such that, if $t\in V_{t_0}$, then
$$|f(\lambda_0+,t)-f(\lambda_0+,t_0)|<\epsilon, \quad\text{ for all  $\lambda_0\in[0,1)$}, $$
that is, $$\lim_{t\to t_0} f(\lambda_0+,t)=f(\lambda_0+,t_0) \quad\text{ for all  $\lambda_0\in[0,1)$}.$$
\end{lemma}
\begin{proof}

 Use the  Moore-Osgood theorem (see \cite[Chapter~VII, Theorem~2]{Graves:1946}) about interchanging limits.
\end{proof}

Example~\ref{ex:levelcont1} allows us to show that, in Lemma~\ref{lem:rightL2}, we cannot drop out  the uniform convergence assumption.
\begin{example}\label{ex:levelcont2}
Take $g:[0,1]\times K\rightarrow \R$ the function defined by
$$g(\lambda,t):=f(t)^-(\lambda)$$
where $f$ is the function in Example~\ref{ex:levelcont1} above. Then
for $\lambda_0=\frac{1}{2}$ and $t_0=0$  we have
\begin{align*}
\lim_{t\to t_0} g(\lambda_0+,t) &=\lim_{t\to 0} \left(\lim_{\lambda\to \frac{1}{2}^+} g(\lambda,t)\right) =
\lim_{t\to 0} \left(\lim_{\lambda\to \frac{1}{2}^+} \left(\lambda-\frac{1}{2}\right)^t \right)=0\\[2ex]
&\neq g(\lambda_0+,t_0)=\lim_{\lambda\to \frac{1}{2}^+} g(\lambda, 0)=1.
\end{align*}

 \end{example}
 \begin{remark}
It is apparent that every $f\in C(K,\fn_{\infty})$ yields two functions, $f_1(\lambda,t)$ and $f_2(\lambda,t)$, which satisfy the conditions in Lemma~\ref{lem:rightL2}, but it is not the case if we consider $f\in C(K,\fn_{\ell})$ (see Example~\ref{ex:levelcont1}).
 \end{remark}

The next lemma tells us that we can get a kind of jointly right limit - continuous property (see Equation~\eqref{eq:RLCP} below) for functions $f:[0,1]\times K \rightarrow
\R$ when we assume that $f(\cdot,t)$ is monotone for every $t\in K$ and $f(\lambda,t)$  converges uniformly on $\lambda$  to $f(\lambda ,t_0)$.
\begin{lemma}\label{lem:jointRL}
Let $f\in \overline{C}([0,1]\times K)$.
Assume that $f(\lambda, t)$ converges to  $f(\lambda, t_0)$, uniformly on $\lambda$ (see~\eqref{eq:uniformLambda}).
Given $(\lambda_0,t_0)\in[0,1)$ and  $\epsilon>0$, there exist $\delta>0$ and an open neighbourhood $V_{t_0}$ of $t_0$ such that
\begin{equation}\label{eq:RLCP}
|f(\lambda,t)-f(\lambda_0+,t_0)|<\epsilon, \quad\text{for all
$(\lambda,t)\in (\lambda_0,\lambda_0+\delta)\times V_{t_0}$.}
\end{equation}
\end{lemma}
\begin{proof}
Fix $(\lambda_0,t_0)\in[0,1)$. Take $\epsilon>0$. We can find $\delta_{(\lambda_0,t_0)}>0$ and $V_{t_0}$ such that
\begin{align}
|f(\lambda+,t_0)-f(\lambda_0+,t_0)|&<\frac{\epsilon}{2},\quad \text{if $\lambda\in (\lambda_0,\lambda_0+\delta_{(\lambda_0,t_0)})$},\\
|f(\lambda_0,t)-f(\lambda_0,t_0)|&<\frac{\epsilon}{2},\quad \text{if $t\in V_{t_0}$},\\
|f(\lambda_0+\delta_{(\lambda_0,t_0)}, t_0)-f(\lambda_0+,t_0)|&<\frac{\epsilon}{2},\quad \text{if $t\in V_{t_0}$.}
\end{align}
Then, take any $(\lambda,t)\in (\lambda_0,\lambda_0+\delta_{(\lambda_0,t_0)})\times V_{t_0}$ and assume that $f(\cdot,t)$ is nondecreasing. We have
\begin{align*}
f(\lambda,t)-f(\lambda_0+,t_0)&\leq f(\lambda_0+\delta_{(\lambda_0,t_0)}, t)-f(\lambda_0+,t_0)\\
&=f(\lambda_0+\delta_{(\lambda_0,t_0)}, t)-f(\lambda_0+\delta_{(\lambda_0,t_0)},t_0)\\
&+f(\lambda_0+\delta_{(\lambda_0,t_0)}, t_0)-f(\lambda_0+,t_0)<\frac{\epsilon}{2}+\frac{\epsilon}{2}=\epsilon.
\end{align*}
On the other hand,
\begin{align*}
f(\lambda,t)-f(\lambda_0+,t_0)&\geq f(\lambda+, t)-f(\lambda_0+,t_0)\\
&=f(\lambda+, t)-f(\lambda+,t_0)+f(\lambda+, t_0)-f(\lambda_0+,t_0)\\
&>-\frac{\epsilon}{2}-\frac{\epsilon}{2}=-\epsilon.
\end{align*}
So we have just shown that $|f(\lambda,t)-f(\lambda_0+,t_0)|<\epsilon$.

The proof assuming that $f(\cdot,t)$ is nonincreasing runs in a similar way, just reversing the inequalities.

\end{proof}

\begin{remark}

{\rm The properties in Lemma~\ref{lem:jointRL} are fullfilled by the corresponding functions $f_1,f_2$ associated to a function $f\in C(K,\fn_{\infty})$.
We saw in Theorem~\ref{th:jointlyLC} that the monotonicity was a sufficient condition to obtain the  jointly left continuous - continuous property, but in order to obtain the corresponding jointly  right limit - continuous property we need something else, namely, the uniform convergence. Functions $f\in C(K,\fn_{\ell})$ lack this condition and we will see in the next example that the jointly right limit - continuous property is not longer true for them.}
\end{remark}

\begin{example}
{\em Take again $g:[0,1]\times K\rightarrow \R$ defined as
$$g(\lambda,t):=f(t)^-(\lambda)$$
where $f$ is the function in Example~\ref{ex:levelcont1} above. We recall the definition of $g(\lambda,t)$:
\begin{align*}
g(\lambda,t)&:=\begin{cases}
0& \text {if $0\leq \lambda \leq \frac{1}{2}$},\\[2ex]
\left(\lambda-\frac{1}{2}\right)^t  & \text {if $\frac{1}{2} < \lambda  \leq 1,$}\\[2ex]
\end{cases}
\end{align*}
for $t\in (0,1]$ and
\begin{align*}
g(\lambda,0)&:=\begin{cases}
0 & \text {if $0\leq \lambda \leq \frac{1}{2}$,}\\[2ex]
1  & \text {if $ \frac{1}{2}<\lambda\leq 1.$}\\[2ex]
\end{cases}
\end{align*}
We are going to prove that it is not jointly right limit - continuous at $(\frac{1}{2},0)$.
Assume the contrary. Then for $\epsilon=\frac{1}{4}$ we can find
$0<\delta<\frac{1}{2}$ and $0<\eta<\frac{1}{2}$ with
$$\left|g(\lambda,t)-g\left(\frac{1}{2}+,0\right)\right|<\frac{1}{4}$$
for all $(\lambda,t)$ with $\frac{1}{2}<\lambda<\frac{1}{2}+\delta$ and $0<t<\eta$.
But, then, by definition of $g(\lambda,t)$, we have
\begin{align*}
\left|\left(\lambda-\frac{1}{2}\right)^t-1\right|&= 1-\left(\lambda-\frac{1}{2}\right)^t <\frac{1}{4},
\end{align*}
which yields
$$\left(\frac{1}{2}\right)^{\frac{1}{2}}>\delta^{\eta}>\left(\lambda-\frac{1}{2}\right)^t>\frac{3}{4},$$
a contradiction.
}
\end{example}

\section{Conclusions}

Just as continuous real-valued functions play a crucial role in classical mathematical analysis, continuous fuzzy-number-valued functions are in the central core of fuzzy mathematical analysis, particularly in several topics such as fuzzy differential equations, fuzzy optimization, fuzzy decision and so on.

It is, thus, apparent that providing a general framework to deal with  continuous fuzzy-number-valued functions is essential to help in the evolution of the above mentioned topics.

In this paper we present a new general setting to deal with level continuous fuzzy-valued functions. Namely, we embed, isometrically and isomorphically,  such functions into a product of spaces of real-valued functions of two variables satisfying certain type of left-continuity, right-continuity and monotonicity.

We consider that our paper can be helpful in classical fuzzy analysis since we provide a wide new framework
to deal with level continuous fuzzy-valued functions by introducing a space of
functions which seems not to have made up its ways in the literature of (not even
necessarily fuzzy) classical analysis yet.
Namely, the originality of our paper relies on the introduction of the so-called
space of left continuous - continuous real-valued functions $\overline{C}([0,1]\times K)$ which has
not been studied before and which is neither a vector space nor even a cone since
it is not closed for the sum.

\section{Compliance with Ethical Standards}

{\bf Author Contributions:} All authors contributed equally and significantly
in writing this paper.

\bigskip
{\bf Funding:} No funding.

\bigskip
{\bf Data Availability Statement:} No data was used for the research described in this
paper.

\bigskip
{\bf Conflicts of Interest:} The authors declare that they have no conflict of interest.

\bigskip
{\bf Ethical approval:} This article does not contain any studies with human participants or animals performed by any of the authors.



\end{document}